\documentclass[11pt]{amsart}
  \baselineskip 1 mm
\usepackage{hyperref}
\hypersetup{
    colorlinks=false, 
    linktoc=all,     
    linkcolor=blue,
        citecolor=red,
    filecolor=black,
    urlcolor=green
}
\usepackage{amssymb}
\usepackage[all]{hypcap}
\usepackage{seqsplit}
\usepackage{bm}
\usepackage{varwidth}
\usepackage{parskip}
\usepackage{enumerate}
\usepackage[inline]{enumitem}
\makeatletter
\newcommand{\inlineitem}[1][]{%
\ifnum\enit@type=\tw@
    {\descriptionlabel{#1}}
  \hspace{\labelsep}%
\else
  \ifnum\enit@type=\z@
       \refstepcounter{\@listctr}\fi
    \quad\@itemlabel\hspace{\labelsep}%
\fi} \makeatother
\newcommand{\subs}{\subset}

\newcommand{\ul}{\underline}

\newcommand{\us}{\underset}
\newcommand{\os}{\overset}

\newcommand{\equ}[1]{%
\begin{equation*}
#1
\end{equation*}
}
\newcommand{\equa}[1]{%
\begin{equation*}
\begin{aligned}
#1
\end{aligned}
\end{equation*}
}

\theoremstyle{plain}
\newtheorem{defn}[equation]{Definition}
\newtheorem{theorem}[equation]{Theorem}

\newtheorem{lemma}[equation]{Lemma}

\begin{document}

\title[An Inequality]{An Inequality}
\author[C.P. Anil Kumar]{C.P. Anil Kumar}
\address{Stat Math Unit, Indian Statistical Institute,
  8th Mile Mysore Road, Bangalore-560059, India}
\email{anilkcp@isibang.ac.in} 
 \maketitle
\begin{abstract}
In this paper we prove that the weighted linear combination of products of the
$k$-subsets of an $n$-set of positive real numbers with weight being
the harmonic mean of their reciprocal sets is less than or equal to
uniformly weighted sum of products of the $k$-subsets with weight
being the harmonic mean of the whole reciprocal set.
\end{abstract}
\maketitle
\section{Introduction}
There is a version of this inequality for $2$-subsets of an $n$-set
which appears in ~\cite{OMPDSL2006} page $327$, Problem $4$ as a
short-listed problem for the Forty Seventh $IMO\ 2006$ held in
Ljubljana, Slovenia. This inequality has an interesting
generalization as stated in the Main
Theorem~\ref{theorem:SumProductSymmetricFunctionInequality} below.
\section{The Main Inequality}
\begin{defn}
Let $A=\{a_1,a_2,\ldots,a_n\}$ be an $n$-set of positive real
numbers. Here we allow the numbers to repeat. i.e. $a_i=a_j$ for
some $1 \leq i \neq j \leq n$. Let $S=\{i_1,i_2,\ldots,i_k\} \subs
\{1,2,3,\ldots,n\}$ be a $k$-subset for some $k \leq n$. Let
$B_S=\{a_{i_1},a_{i_2},a_{i_3},\ldots, a_{i_k}\} \subs A$ be its
corresponding set. The reciprocal set denoted by $B_S^{-1}$ of the
set $B_S$ is defined to be the set
$B_S^{-1}=\{\frac{1}{a_{i_1}},\frac{1}{a_{i_2}},\frac{1}{a_{i_3}},\ldots,
\frac{1}{a_{i_k}}\}$.
\end{defn}
Now we state the main theorem.
\begin{theorem}
\label{theorem:SumProductSymmetricFunctionInequality} Let $[\ul{n}]
= \{1,2,3,\ldots,n\}$ denote the set of first $n$ natural numbers.
Let $A=\{a_i: i=1, \ldots,n\}$ be an $n$-set of positive real
numbers. For any subset $S \subs [\ul{n}]$, let $\prod a_S$ denote
$\us{i \in S}{\prod} a_i$ and $\sum a_S$ denote $\us{i \in S}{\sum}
a_i$. Then
\begin{equation}
\us{k-subset\ S \subs [\ul{n}]}{\sum}\bigg(\frac{\prod a_S}{\sum
a_S}\bigg) \leq \frac{n}{k} \bigg(\frac{\us{k-subset\ S \subs
[\ul{n}]}{\sum} \prod a_S}{\sum a_{[\ul{n}]}}\bigg)
\end{equation}
i.e. The weighted linear combination of products of the
$k$-sets with weight as harmonic mean of their reciprocal sets  is
less than or equal to uniformly weighted sum of products of the $k$-
sets with weight the harmonic mean of the whole reciprocal set.

We also observe that the equality occurs if and only if $a_1=a_2=\ldots=a_n$.
\end{theorem}

\section{Some Simple Cases of the General Inequality}
We prove a few lemmas.
\begin{lemma}
\label{lemma:ReciprocalTriples} Let $a_1,a_2,a_3$ be three positive
real numbers. Then the sum of the reciprocals of $a_i$ is greater
than or equal to sum of the reciprocals of their pairwise averages.
\begin{equation}
\frac{1}{a_1}+\frac{1}{a_2}+\frac{1}{a_3} \geq
\frac{1}{\frac{a_1+a_2}{2}} + \frac{1}{\frac{a_2+a_3}{2}} +
\frac{1}{\frac{a_3+a_1}{2}}
\end{equation}
\end{lemma}
\begin{proof}
We have from $AM-HM$ inequality applying to the reciprocals
$\frac{1}{a_1},\frac{1}{a_2},\frac{1}{a_3}$ we get
\equa{\frac{\frac{1}{a_1}+\frac{1}{a_2}}{2} &\geq \frac{2}{a_1+a_2}\\
\frac{\frac{1}{a_2}+\frac{1}{a_3}}{2} &\geq \frac{2}{a_2+a_3}\\
\frac{\frac{1}{a_3}+\frac{1}{a_1}}{2} &\geq \frac{2}{a_3+a_1}} and
adding these inequalities we have
\begin{equation*}
\frac{1}{a_1}+\frac{1}{a_2}+\frac{1}{a_3} \geq
\frac{1}{\frac{a_1+a_2}{2}} + \frac{1}{\frac{a_2+a_3}{2}} +
\frac{1}{\frac{a_3+a_1}{2}}
\end{equation*}
Hence the Lemma~\ref{lemma:ReciprocalTriples} follows.
\end{proof}
\begin{lemma}
\label{lemma:MainThree}
Let $a_1,a_2,a_3$ be three positive real
numbers. Then
\begin{equation}
\frac{a_1a_2}{a_1+a_2}+\frac{a_2a_3}{a_2+a_3}+\frac{a_3a_1}{a_1+a_3}
\leq \frac{3(a_1a_2+a_2a_3+a_3a_1)}{2(a_1+a_2+a_3)}
\end{equation}
\end{lemma}
\begin{proof}
In order to prove Lemma~\ref{lemma:MainThree} first we make a
simplification by assuming without loss of generality that
$a_1+a_2+a_3=1$. This can be done by normalizing with $a_1+a_2+a_3$.
Now
\equa{&2\frac{a_1a_2}{a_1+a_2}+2\frac{a_2a_3}{a_2+a_3}+2\frac{a_3a_1}{a_1+a_3}-2(a_1a_2+a_2a_3+a_3a_1)\\
&=a_1a_2a_3\bigg(\frac{1}{\frac{a_1+a_2}{2}} +
\frac{1}{\frac{a_2+a_3}{2}}
+ \frac{1}{\frac{a_3+a_1}{2}}\bigg)\\
&\leq a_1a_2a_3\bigg(\frac{1}{a_1}+\frac{1}{a_2}+\frac{1}{a_3}\bigg)
\text{ Using Lemma~\ref{lemma:ReciprocalTriples} }\\
&=a_1a_2+a_2a_3+a_3a_1} Hence the Lemma~\ref{lemma:MainThree}
follows.
\end{proof}
\begin{lemma}
\label{lemma:Reciprocal} Let $a_1,a_2,a_3,\ldots,a_n$ be $n$
positive real numbers. Then the sum of the reciprocals of $a_i$ is
greater than or equal to the sum of the reciprocals of their
$(n-1)$-wise averages.
\begin{equation}
\label{eq:Reciprocal}
\begin{aligned}
\frac{1}{a_1}+\frac{1}{a_2}+\frac{1}{a_3}+ \ldots + \frac{1}{a_n} & \geq \frac{1}{\frac{a_1+a_2+\ldots+a_{n-1}+a_n-a_n}{n-1}}+ \frac{1}{\frac{a_1+a_2+a_3+\ldots+a_n-a_{n-1}}{n-1}}\\
& + \frac{1}{\frac{a_1+a_2+a_3+ \ldots + a_n-a_{n-2}}{n-1}} + \ldots
+ \frac{1}{\frac{a_1+a_2+a_3+\ldots+a_n-a_{1}}{n-1}}
\end{aligned}
\end{equation}
\end{lemma}
\begin{proof}
This is a generalization of Lemma~\ref{lemma:ReciprocalTriples} to
$n$-positive real numbers $a_1,a_2,\ldots,a_n$. We have from $AM-HM$
inequality applying to the reciprocals
$\frac{1}{a_1},\frac{1}{a_2},\frac{1}{a_3},\ldots,\frac{1}{a_n}$ we
get the following set of inequalities. For every $1 \leq j \leq n$,
we get \equ{\frac{1}{(n-1)}\bigg(\us{i \neq j,1=1}{\os{n}{\sum}}
\frac{1}{a_i}\bigg) \geq \frac{(n-1)}{\us{i \neq j,
i=1}{\os{n}{\sum}} a_i}} and adding these inequalities we have
\begin{equation*}
\begin{aligned}
\frac{1}{a_1}+\frac{1}{a_2}+\frac{1}{a_3}+ \ldots + \frac{1}{a_n} & \geq \frac{1}{\frac{a_1+a_2+\ldots+a_{n-1}+a_n-a_n}{n-1}}+ \frac{1}{\frac{a_1+a_2+a_3+\ldots+a_n-a_{n-1}}{n-1}}\\
& + \frac{1}{\frac{a_1+a_2+a_3+ \ldots + a_n-a_{n-2}}{n-1}} + \ldots
+ \frac{1}{\frac{a_1+a_2+a_3+\ldots+a_n-a_{1}}{n-1}}
\end{aligned}
\end{equation*}
Hence the Lemma~\ref{lemma:Reciprocal} follows.
\end{proof}
\begin{lemma}
Let $[\ul{n}] = \{1,2,3,\ldots,n\}$ denote the set of first $n$
natural numbers. Let $A=\{a_i: i=1,\ldots, n\}$ be a set of $n$
positive real numbers. Then
\begin{equation}
\us{(n-1)-subset\ S \subs [\ul{n}]}{\sum}\bigg(\frac{\prod a_S}{\sum
a_S}\bigg) \leq \frac{n}{(n-1)} \bigg(\frac{\us{(n-1)-subset\ S
\subs [\ul{n}]}{\sum} \prod a_S}{\sum a_{[\ul{n}]}}\bigg)
\end{equation}
\end{lemma}
\begin{proof}
This is a generalization of the above Lemma~\ref{lemma:MainThree} to
the case of $(n-1)$-subsets of an $n$-set. The proof is similar to
the proof of Lemma~\ref{lemma:MainThree} except here we use
Lemma~\ref{lemma:Reciprocal} instead of
Lemma~\ref{lemma:ReciprocalTriples}.
\end{proof}
\begin{lemma}
\label{lemma:MainTwoSets} Let $A=\{a_i: i=1,\ldots,n\}$ be a set of
$n$-positive real numbers. Then
\begin{equation}
\us{i<j}{\sum}\frac{a_ia_j}{a_i+a_j} \leq
\frac{n}{2(a_1+a_2+\ldots+a_n)} \us{i<j}{\sum}a_ia_j
\end{equation}
\end{lemma}
\begin{proof}
The proof is as follows. Again by normalizing with
$\us{i=1}{\os{n}{\sum}} a_i$ we can assume that
$\us{i=1}{\os{n}{\sum}} a_i=1$ and it is enough to prove that
\begin{equation}
\us{i<j}{\sum}\frac{a_ia_j}{a_i+a_j} \leq \frac{n}{2}
\us{i<j}{\sum}a_ia_j
\end{equation}
So consider
\equa{&2\us{i<j}{\sum}\frac{a_ia_j}{a_i+a_j}-2\us{i<j}{\sum}a_ia_j
=\us{i<j}{\sum}\frac{2}{a_i+a_j}\bigg(\us{k \neq i,k \neq
j}{\sum}a_ia_ja_k\bigg)\\
&=\bigg(\us{3-subset\ S \subs [\ul{n}]}{\sum} \bigg(\us{2-subset\ T
\subs S}{\sum} \frac{2\prod a_S}{\sum a_T}\bigg)\bigg)\\
&\text{ Now using Lemma~\ref{lemma:ReciprocalTriples} for all } 3
\text{ subsets of }\{1,2,3,\ldots,n\} \text{ we get } \\
&\leq (n-2)\bigg(\us{2-subset\ S \subs [\ul{n}]}{\sum} \prod
a_S\bigg)} Hence the lemma follows.
\end{proof}

\section{Proof of the Main Theorem}
Here we prove the Main
Theorem~\ref{theorem:SumProductSymmetricFunctionInequality}
\begin{proof} Now we generalize to the case given in the
Theorem~\ref{theorem:SumProductSymmetricFunctionInequality} by first
normalizing the inequality with $\sum a_i$ so that we can assume
that $\sum a_i=1$. And we have to proof the following inequality
\begin{equation*}
\us{k-subset\ S \subs [\ul{n}]}{\sum}\bigg(\frac{\prod
a_S}{\frac{\sum a_S}{k}}\bigg) \leq n\bigg(\us{k-subset\ S \subs
[\ul{n}]}{\sum} \prod a_S\bigg)
\end{equation*}
We have
\begin{equation*}
\begin{aligned}
&\us{k-subset\ S \subs [\ul{n}]}{\sum}\bigg(k\frac{\prod a_S}{\sum a_S}\bigg)-k\bigg(\us{k-subset\ S \subs [\ul{n}]}{\sum} \prod a_S\bigg)\\
&= k\bigg(\us{k-subset\ S \subs [\ul{n}]}{\sum} \frac{\prod a_S(1-\sum a_S)}{\sum a_S}\bigg)\\
&=k\bigg(\us{(k+1)-subset\ S \subs [\ul{n}]}{\sum} \bigg(\us{k-subset\ T \subs S}{\sum} \frac{\prod a_S}{\sum a_T}\bigg)\bigg)\\
&\leq(n-k)\bigg(\us{k-subset\ S \subs [\ul{n}]}{\sum} \prod
a_S\bigg) \text{ Using inequality~\ref{eq:Reciprocal} in
Lemma~\ref{lemma:Reciprocal}}
\end{aligned}
\end{equation*}

The equality occurs when if and only all the AM-HM inequalities involved give equality which holds if and only if 
$a_1=a_2=\ldots=a_n$.
Hence the Main
Theorem~\ref{theorem:SumProductSymmetricFunctionInequality} follows.
\end{proof}

\section{Acknowledgements}
The author likes to thank Prof. C.R. Pranesachar, HBCSE, TIFR,
Mumbai for mentioning this problem on inequalities given in
Lemma~\ref{lemma:MainTwoSets} which the author was able to suitably
generalize to the Main
Theorem~\ref{theorem:SumProductSymmetricFunctionInequality} in this
article.

\bibliographystyle{abbrv}

\begin{thebibliography}{1}
\bibitem{OMPDSL2006}
D.~Dujukic,V.~Janokovic,I.~Matic,N.~Petrovic. {The IMO Compendium},
Problem Books in Mathematics.
\end{thebibliography}
\def\cprime{$'$}

\end{document}